\documentclass[11pt]{amsart}
\usepackage{amsmath}
\usepackage{amsthm}
\usepackage{amsfonts}
\usepackage{amssymb}
\usepackage[all]{xy}

\newcommand{\C}{\mathcal C}
\newcommand{\X}{\mathcal X}
\newcommand{\Y}{\mathcal Y}
\newcommand{\Z}{\mathcal Z}
\newcommand{\mb}{\mathbb}
\newcommand{\ze}{\mathbb{Z}}
\newcommand{\ol}{\overline}
\newcommand{\SE}{\mathcal{SE}}
\newcommand{\E}{\mathcal E}

\newcommand{\ra}{\rightarrow}
\newcommand{\lra}{\longrightarrow}
\newcommand{\co}{\mathbb C}
\newcommand{\col}{\colon}

\newcommand{\D}{\text{Def}}
\newcommand{\A}{\text{Aut}}
\newcommand{\Sgb}{\overline{S_g}}
\newcommand{\Mgb}{\overline{M_g}}

\renewcommand{\ol}{\overline}
\renewcommand{\phi}{\varphi}
    
    \newtheorem{Lem}{Lemma}[section]
    \newtheorem{Prop}[Lem]{Proposition}

    \newtheorem{Thm}[Lem]{Theorem}

\theoremstyle{definition}

    \newtheorem{Def}[Lem]{Definition}
    \newtheorem{Exa}[Lem]{Example}
    \newtheorem{Rem}[Lem]{Remark}
    \newtheorem{Not}[Lem]{Notation}

    \DeclareMathOperator{\Pic}{Pic}

\begin{document}

\title{Enriched  spin curves on stable curves with two components}

\author{Marco Pacini}
\thanks{The author was partially supported by CNPq (Proc.151610/2005-3) and by Faperj (Proc.E-26/152-629/2005)\\
M.S.C. (2000): Primary 14H10 Secondary 14K30}

\begin{abstract}
In \cite{M}, Main\`o  constructed a moduli space for enriched stable curves, by blowing-up the moduli space of Deligne-Mumford stable curves. We introduce enriched spin curves, showing that a parameter space for these objects is obtained by blowing-up the moduli space of spin curves.

\medskip

\end{abstract}

\maketitle

\section{Introduction}

A basic tool in the theory of limit linear series is to consider degenerations of smooth curves to singular ones. In \cite{EH}, Eisenbud and Harris developed a theory for curves of compact type, i.e. curves having only separating nodes. The advantage to work with curves of compact type is the following. Let $B$ be the spectrum of a DVR and $f\col\C\ra B$ be a general smoothing of a nodal curve $C$, i.e. $C=f^{-1}(0)$ for some $0\in B$ and $f^{-1}(b)$ is a smooth curve for $b\ne 0$ and $\C$ is smooth. If $C_1\dots,C_\gamma$ are the components of $C$, then all the extensions of a line bundle $\mathcal L^*$ over $f^{-1}(B-0)$ are given by $\mathcal L\otimes\mathcal O_{\C}(C_i)$, where $\mathcal L$ is a fixed extension. If $C$ is of compact type, then $\mathcal L\otimes\mathcal O_{\C}(C_i)$ does not depend on the smoothing. This is not true in general and it is the main difficulty arising when one tries to extend the theory to a more general class of curves. The problem was solved in \cite{EM} for general curves with two components, but a general analysis is still not available.

With these motivations, the notion of \emph{enriched stable curve} is introduced in \cite{M}. Let $C$ be a stable curve with components $C_1\dots, C_\gamma$.  An \emph{enriched stable curve of $C$} is given by $(C, \mathcal O_{\C}(C_1)|_{C},\dots, \mathcal O_{\C}(C_\gamma)|_{C})$, where $f\col\C\ra B$ is a general smoothing of $C$. Necessarily, we have $\otimes_i^\gamma \mathcal O_{\C}(C_i)|_C\simeq\mathcal O_C$. In \cite{M}, it is shown that an enriched stable curve of $C$ only depends on the first order deformation of the given smoothing $\C$ of $C$. Furthermore it is possible to understand when two first order deformations of $C$ give rise to the same enriched stable curve. A moduli space for enriched stable curves is constructed by taking blow-ups of the base of the universal deformation of stable curves and glueing all these blow-ups together.

On the other hand for a given family of nodal curves $f\col\C\ra B$ and a line bundle 
$\mathcal N$ of $\C$ of relative even degree, viewed as a family of line bundles on the fibers of $f$, one can consider the problem of compactifying the moduli space for roots of the restriction of $\mathcal N$ to the fibers of $f$. In \cite{CCC}, a moduli space is constructed in terms of \emph{limit square roots}. In particular, when $f\col\C\ra B$ is a stable family and $\mathcal N=\omega_f$, this moduli space represents \emph{spin curves of stable curves}, a generalization of theta characteristics on smooth curves. 
In \cite{C}, a moduli space $\Sgb$ for spin curves of stable curves of genus $g$ is constructed. The moduli space $\Sgb$ is endowed with a natural finite morphism $\phi\col\Sgb\lra\Mgb$ onto the moduli space of Deligne--Mumford stable curves. As one can expect, the degree of $\phi$ is $2^{2g}$. The fibers of $\phi$ over represent \emph{spin curves}. The paper \cite{CC} provides an explicit combinatorial description of the boundary.

Since a parameter space for enriched curves is obtained by blowing-up $\Mgb$, we expect that a point of a blow-up of $\Sgb$ parametrizes roots of all the possible degenerations of the dualizing sheaf on families of stable curves. Indeed, let $C$ be a stable curve. A curve $X$ is obtained from $C$ by \emph{blowing-up} a subset $\Delta$ of the set of the nodes of $C$ if there is a morphism $\pi\col X\ra C$ such that, for every $p_i\in\Delta,$ $\pi^{-1}(p_i)=E_i\simeq\mb{P}^1$ and $\pi\col X-\cup_i E_i\ra C-\Delta$ is an isomorphism.  The curves $E_i$ are called \emph{exceptional}. Let $C$ be with two smooth components, $C_1,C_2$. We define an \emph{enriched spin curve supported on $X$} as a tern $(X, L_1, L_2)$, where $X$ is a blow-up of $C$ at a \underline{proper} subset of nodes, and $L_1,L_2$ are line bundles of $X$ such that: 
\begin{itemize}
\item[(i)]
$L_i$ has degree one on exceptional components of $X$;
\item[(ii)]
if $\widetilde{X}$ is the complement of the union of the exceptional components of $X$, then  $(L_i|_{\widetilde{X}})^{\otimes 2}\simeq \omega_{\widetilde{X}}\otimes\mathcal O_{\widetilde{\X}}(C_i)|_{\widetilde{X}}$, for $i=1,2$, where $\widetilde{\X}\ra B$ is a general smoothing of $X$, and $(L_1)|_{\widetilde{X}}\otimes (L_2)|_{\widetilde{X}}\simeq\omega_{\widetilde{X}}$. 
\end{itemize}

We introduce a notion of isomorphism between enriched spin curves and we denote by $\ol{\SE_C}$ the set of isomorphism classes of enriched spin curves and by 
$\SE_X$ the subset of the ones supported on $X$.  In Lemma \ref{desing}, we show when $\Sgb$ is singular at a spin curve $\xi$ of a curve $C$ with two smooth components.  A detailed analysis of the singular locus of $\Sgb$ is given in \cite{L}.  We consider a distinguished subset $D_X$ of $\Sgb$ containing $\xi$ as singular point and we find a blow-up $D^\nu_X\ra D_X$, desingularizing $D_X$, with exceptional divisor $\mb{P}^{\delta-1}_\xi$, where $\delta$ is the number of nodes of $C$. The following theorem sums-up Proposition \ref{tors} and Theorem \ref{Th1}, \ref{Th2}.

\begin{Thm}\label{main} 
Let $C$ be with $\delta$ nodes and two smooth components of genus at least 1. Assume that $\text{Aut}(C)=\{id\}.$ 
For $\xi$ running over the set of spin curves of $C$ which are singular points of $\Sgb$, there exist $\delta$ hyperplanes $H_{\xi, 1},\dots, H_{\xi, \delta}$ of $\mb{P}^{\delta-1}_\xi$,  such that:  
\begin{itemize}
\item[(i)]
$\SE_C$  and $\cup_\xi (\mb{P}^\delta_\xi -(\cup_{1\le i\le \delta} H_{\xi, i}))$ are isomorphic torsors;
\item[(ii)] 
 if $X_I$ and $\widetilde{X_I}$ are the blow-up and the normalization of $C$ at a subset $I=\{p_1,\dots, p_h\}$ of nodes of $C$, with $1\le h<\delta$, then the set of the isomorphism classes of enriched spin curves of $C$ supported on $X_I$, $\SE_{\widetilde{X_I}}$ and $\cup_\xi (\cap_{1\le i\le  h} H_{\xi, i}-\cup_{h<i\le \delta} H_{\xi, i})$ are isomorphic torsors.
 \end{itemize}
\end{Thm}

The proof of the Theorem \ref{main} uses some ideas of \cite{P}.
We see that $\ol{\SE_C}$ is parametrized by a complete variety and that it is stratified in terms of enriched spin curves of partial normalizations of $C$, as illustrated in Example \ref{Example}. 
Furthermore, recall that the moduli space of enriched stable curve constructed in \cite{M} is not complete. The analysis of this paper suggests that a compactification of this moduli space could be given in terms of enriched stable curves on partial normalizations of stable curves.

Although the hypothesis that the components of $C$ are smooth could be removed 
 in Theorem \ref{main}, the combinatorics involved became a bit harder especially in the proof of Theorem \ref{Th2}. Therefore, we choose to present the simplest case in this paper and we plan to investigate the problem of the generalization to any stable curve in a different paper.

\medskip

We will use the following notation and terminology. We work over the field of complex numbers. A \emph{curve} is a connected projective curve which is Gorenstein and reduced. A \emph{stable} (\emph{semistable}) curve $C$ is a nodal curve such that every smooth rational subcurve of $C$ meets the rest of the curve in at least $3$ points ($2$ points). Let $\omega_X$ be the dualizing sheaf of a curve $X$. The genus of $X$ is $g=h^0(X,\omega_X).$ If $Z\subset X$ is a subcurve, set $Z^c:=\overline{X-Z}.$ 
A \emph{family of curves} is a proper and flat morphism $f\col\mathcal W\ra B$ whose fibers are curves. We denote either by $\omega_f$ or by $\omega_{\mathcal W/B}$, the relative dualizing sheaf of a family. A \emph{smoothing} of a curve $X$ is a family $f\col\mathcal X\ra B,$ where $B$ is a smooth, connected, affine curve of finite type, with a distinguished point $0\in B,$ such that $X=f^{-1}(0)$ and $f^{-1}(b)$ is smooth for $b\in B-0.$ A \emph{general smoothing} is a smoothing with smooth total space. 
A curve $X$ is obtained from $C$ by \emph{blowing-up} a subset $\Delta$ of the set of the nodes of $C,$ if there is a morphism $\pi\col X\ra C$ such that, for every $p_i\in\Delta,$ $\pi^{-1}(p_i)=E_i\simeq\mb{P}^1$ and $\pi\col X-\cup_i E_i\ra C-\Delta$ is an isomorphism. For every $p_i\in\Delta,$ we call $E_i$ an \emph{exceptional component}. 
If $X$ is a curve , we denote by $\A(X)$ the group of automorphisms of $X$.

\section{The moduli space of spin curves}

In \cite{CCC}, the authors described compactifications of moduli spaces of roots of line bundles on smooth curves, in terms of \emph{limit square roots}.

Let $C$ be a nodal curve and let $N\in\text{Pic}(C)$ be of even degree.
A tern $(X,L, \alpha),$ where $\pi\col X\ra C$ is a blow-up of $C,$ $L$ is a line bundle on $X$ and $\alpha$ is a homomorphism $\alpha\col L^{\otimes 2}\ra \pi^*(N),$ is a \emph{limit square root} of $(C,N)$ if:
\begin{itemize}
\item[(i)]
the restriction of $L$ to every exceptional component has degree $1;$
\item[(ii)]
the homomorphism $\alpha$ is an isomorphism at the points of $X$ not belonging to an exceptional component;
\item[(iii)]
for every exceptional component $E$ such that $E\cap E^c=\{p,q\}$ the orders of vanishing of $\alpha$ at $p$ and $q$ add up to $2.$ 
\end{itemize}

The curve $X$ is called the \emph{support} of the limit square root. If $C$ is stable, then a limit square root of $(C, \omega_C)$ is said to be a \emph{spin curve of} $C$.

If $X$ is a blow-up of a nodal curve $C$, denote by $\widetilde{X}:=\overline{X-\cup E},$ where $E$ runs over the set of the exceptional components.  There exists a notion of isomorphism of limit square roots. By \cite[Lemma 2.1]{C}, two limit square roots $\xi=(X,L,\alpha)$ and $\xi'=(X,L',\alpha')$ are isomorphic if and only if the restrictions of $L$ and $L'$ to $\widetilde{X}$ are isomorphic. Denote by $\A(\xi)$ the group of automorphisms of $\xi$. A limit square root of $(C, N)$ supported on a blow-up $X$ with exceptional components $\{E_i\}$ is determined by the line bundle $L$ obtained by glueing  $\mathcal O_{E_i}(1)$, for every $E_i$, and a square root of $(\pi^*N)|_{\widetilde{X}}(\sum(- p_i-q_i))$, where $\{p_i, q_i\}=E_i\cap E_i^c$. Indeed, it is possible to define a homomorphism $\alpha$ such that $(X, L, \alpha)$ is a limit square root. In the sequel, if no confusion may arise, we denote a limit square root simply by $(X, L)$. Let $f\col\mathcal C\ra B$ be a family of nodal curves over a quasi-projective scheme $B$ and let $\mathcal N\in\Pic(\mathcal C)$ be of even relative degree. There exists a quasi-projective scheme $\overline{S}_f(\mathcal N),$ finite over $B$, which is a coarse moduli space, with respect to a suitable functor, of isomorphism classes of limit square roots of the restriction of $\mathcal N$ to the fibers of $f$. For more details, we refer to \cite[Theorem 2.4.1]{CCC}.

Let $C$ be a nodal curve and $N\in\Pic (C)$ of even degree. Denote by $\overline{S}_C(N)$ the zero-dimensional scheme $\overline{S}_{f_C}(N),$ where $f_C\col C\ra\{pt\}$ is the trivial family. In particular, $\overline{S}_C(N)$ is in bijection with the isomorphism classes of limit square roots of $(C, N)$. If $f\col\mathcal C\ra B$ is a family of curves and $\mathcal N\in\Pic\mathcal C,$ then the fiber of $\overline{S}_f(\mathcal N)\ra B$ over $b\in B$ is $\overline{S}_{f^{-1}(b)}(\mathcal N|_{f^{-1}(b)}),$ as explained in \cite[Remark 2.4.3]{CCC}. 
Denote by $\Sigma_X$ the graph having the connected components of $\widetilde{X}$ as vertices and the exceptional components as edges.
By \cite[4.1]{CCC}, the multiplicity of $\overline{S}_C(N)$ in $\xi=(X,G,\alpha)$ is $2^{b_1(\Sigma_X)}$.

In \cite{C}, the author constructed the moduli space $\Sgb$ of spin curves of stable curves of genus $g$. The moduli space $\Sgb$ is endowed with a finite map $\phi\col\Sgb\ra \Mgb$, of degree $2^{2g}$. Let $\overline{M_g^0}$ be the open subset parametrizing curves without non-trivial automorphisms and let $\overline{S_g^0}$ be the restriction of $\Sgb$ over $\overline{M_g^0}$. In this case, if $f\col\mathcal C\ra B$ is a family of stable curves with moduli morphism $B\ra\overline{M_g^0}$, then $\overline{S}_f(\omega_f)=\Sgb\times_{\Mgb} B$.

\begin{Not}\label{DX}
Let $C$ be a stable curve with $\A(C)=\{id\}$ and nodes $p_1,\dots, p_\delta$.  Let $\D(C)$ be the base of the universal deformation of $C$, which is a $(3g-3)$-dimensional polydisc in $\co^{3g-3}_{t_1,\dots,t_{3g-3}}$. Here 
$\{t_i=0\}$ is the locus where the node $p_i$ is preserved. In particular, locally analytically at $C$, we have $\D(C)\subset \Mgb$. Denote by $D_C=\D(C)\cap\co^\delta_{t_1,\dots, t_\delta}$ and by $D_X=\phi^{-1}(D_C)$, where $\phi\col\Sgb\ra\Mgb$.
\end{Not}

\begin{Lem}\label{desing}
Let $C$ be a stable curve with two smooth components and $\delta$ nodes with $\A(C)=\{id\}$. Let $\xi$ be a spin curve of $C$. Then, $\Sgb$ is singular at $\xi$ if and only if $\xi$ is supported on the blow-up at the whole set of nodes of $C$. In this case, locally analytically at $\xi$, the equations of $D_X$ are of type: 
\begin{equation}\label{spin-equat}
w_{ii}w_{jj}=w_{ij}^2,\,\,\,w_{ii}w_{jj}w_{kk}=w_{ij}w_{jk}w_{ik}, \text{ for }1\le i<j<k\le \delta.
\end{equation}
The blow-up $D^\nu_X$ of $D_X$ at the ideal $(w_{11},w_{12},w_{13},\dots,w_{1\delta})$ is smooth.
\end{Lem}

\begin{proof}
Keep Notation \ref{DX}. Let $\xi$ be a spin curve of $C$ supported on the blow-up $X$ of $C$ at  the nodes $p_1,\dots, p_h$. Let $\rho\col D_C\ra D_C$ be given by: 
$$\rho(t_1,\dots,t_h,t_{h+1},\dots,t_\delta)=(t_1^2,\dots,t_h^2,t_{h+1},\dots,t_\delta).$$
We have that $\A(\xi)$ acts on $D_C$ as subgroup of the group of automorphisms of $D_C$, commuting with $\rho$, as follows. If $h<\delta$, then 
$\Sigma_X$ is a graph with one node and $h$ loops. Thus $\A(\xi)=\{id\}$ by 
\cite[Lemma 2.3.2, Lemma 3.3.1]{CCC} and hence $\Sgb$ is smooth at $\xi$.
If $h=\delta$, then $\Sigma_X$ is a graph with two nodes and $h$ edges. Again by \cite[Lemma 2.3.2, Lemma 3.3.1]{CCC}, we have: $$\A(\xi)=\{id, (t_1,\dots, t_\delta)\stackrel{\beta}{\ra} (-t_1,\dots, -t_\delta)\}.$$ By definition, $D_X=D_C/\A(\xi)$. If we set $w_{ij}=t_i t_j$ for $1\le i\le j\le \delta$, then locally analytically at $\xi$, the equations of $D_X$ are as in (\ref{spin-equat}). Now, $\Sgb$ is given by $D_X\times (\D(C)\cap\co^{3g-\delta-3}_{t_{\delta+1},\dots,t_{3g-3}})$ and $\Sgb$ is singular at $\xi$. 

Let $D^\nu_X$ be the blow-up of $D_X$ at the ideal $(w_{11},w_{12},\dots,w_{1\delta})$. Cover $D^\nu_X$ with $\delta$ open subsets $U_1,U_2,\dots, U_\delta$, such that the equation of $U_s$ is: 

$$\begin{cases}
\begin{array}{ll}
w_{1i}=\alpha_{is} w_{1s} & 1\le i\le \delta \text{ for } i\ne s \\
w_{ii}w_{jj}=w_{ij}^2 & 1\le i<j\le \delta \\
w_{ii}w_{jj}w_{kk}=w_{ij}w_{jk}w_{ik} & 1\le i<j<k\le \delta
\end{array}
\end{cases}$$ 

for every $s=1,\dots,\delta$. After few calculations, we get:

$$\begin{cases}
\begin{array}{ll}
w_{is}=\alpha_{is} w_{ss} & 1\le i< s \\
w_{si}=\alpha_{is} w_{ss} & s< i\le\delta \\
w_{ij}=\alpha_{is}\alpha_{js} w_{ss} & 1\le i\le j\le \delta \text{ for } i,j\ne s \\
\end{array}
\end{cases}$$ 

In particular, $U_s$ is smooth for every $s$, hence $D^\nu_X$ is smooth.  
\end{proof}

\begin{Rem}\label{line-torsor} 
Keep the notation of Lemma \ref{desing}, with $D_X$ singular. Consider the map $\phi\col D_X\ra D_C$.
 Of course, $\phi$ is a finite map of degree $2^{\delta-1}$, ramified over the coordinate hyperplanes of $D_C$. Let $R\subset D_C$ be a line away from the coordinate hyperplanes and containing the origin. By construction, $\phi^{-1}(R)$ is a union of $2^{\delta-1}$ lines of $D_X$ through the origin, intersecting transversally. The group  of the automorphisms of $D_X$ commuting with $\phi$ is isomorphic to 
 $(\ze/2\ze)^{\delta-1}$ and acts freely and transitively on the set of $2^{\delta-1}$ lines.  Let $\nu\col D_X^\nu\ra D_X$ be as in Lemma \ref{desing} and let $\mb{P}_\xi^{\delta-1}=\nu^{-1}(0)$ be the exceptional divisor over the origin. The pull-back to $D_X^\nu$ of a line of $D_C$ is a disjoint union of lines, intersecting $\mb{P}^{\delta-1}_\xi$. Let $H_{\xi, i}\subset\mb{P}_\xi^{\delta-1}$, for $i=1,\dots,\delta$, be the hyperplane such that the pull-back to $D^\nu_X$ of a line contained in $\{t_i=0\}\subset D_C$ intersects $\mb{P}^{\delta-1}_\xi$ in $H_{\xi, i}$. We see that $\mb{P}_\xi^{\delta-1}-\cup_{1\le i\le\delta} H_{\xi,i}$ is a $(\ze/2\ze)^{\delta-1}\times(\co^*)^{\delta-1}$-torsor. Similarly, for every $\emptyset\ne I\subsetneq\{1,\dots,\delta\}$, we have that $\cap_{i\in I} H_{\xi,i}-\cup_{i\notin I} H_{\xi,i}$ is a $(\ze/2\ze)^{\delta-|I|-1}\times(\co^*)^{\delta-|I|-1}$-torsor.
\end{Rem}

\section{Enriched spin curves}\label{3}

In \cite{M},  an \emph{enriched stable curve} of a stable curve $C$ with irreducible components $C_1,\dots, C_\gamma$ is defined as $(C, T_{C_1},\dots, T_{C_\gamma})$, where $T_{C_i}=\mathcal O_{\C}(C_i)|_C$ and $\C$ is a general smoothing of $C$.  The line bundle $T_{C_i}$ is called a \emph{twister induced by  $C_i$ and $\C$}. Let  $\E_C$ be the set of the enriched stable curves of $C$. Let $D_C$ be as in Notation \ref{DX}.  In the following Lemma, we see that one can obtain a parameter space for $\E_C$, by taking a blow-up of $D_C$

\begin{Prop}\label{Maino}
Let $C$ be a stable curve with $\delta$ nodes and two smooth components $C_1$ and $C_2$. Then $\E_C$ forms a $(\co^*)^{\delta-1}$-torsor, which is isomorphic to the $(\co^*)^{\delta-1}$-torsor of linear directions in $D_C$ through the origin, away from the coordinate hyperplanes. The enriched curve  corresponding to a line $R\subset D_C$ is $(C, T_{C_1}, T_{C_2})$, where 
$T_{C_1}$ (resp. $T_{C_2}$) is the twister induced by $C_1$ (resp. $C_2$) and any general smoothing $\C\ra B$ of $C$ such that, up to restrict $B$, the induced map $B\ra\D(C)$ has $R$ as image.
\end{Prop}

For a proof of Proposition \ref{Maino}, see \cite[Proposition 3.4, 3.9]{M}.  
We will need the following result, characterizing the tuples of line bundles which are twisters.

\begin{Prop}\label{enr-char}
Let $C$ be a stable curve with irreducible components $C_1,\dots, C_\gamma$ 
and let $T_1,\dots,T_\gamma$ be line bundles on  $C$. Then $(C, T_1,\dots, T_\gamma)$ is an enriched stable curves of $C$ if and only if the following conditions are satisfied:
\begin{itemize}
\item[(i)]
$T_i\otimes\mathcal O_{C_i}\simeq \mathcal O_{C_i}(-p_{i,1}-\dots -p_{i,n_i})$ and  $T_i\otimes\mathcal O_{C^c_i}\simeq \mathcal O_{C^c_i}(p_{i,1}+\dots +p_{i,n_i})$ for every $i=1,\dots,\gamma$, where $\{p_{i,1},\dots,p_{i,n_i}\}=C_i\cap C^c_i$. 
\item[(ii)]
$\otimes_{i=1}^\gamma T_i\simeq\mathcal O_C$.
\end{itemize}
\end{Prop}

For a proof of Proposition \ref{enr-char}, see \cite[Proposition 3.16]{M} or \cite[Theorem 6.10]{EM}. Similarly, we introduce enriched spin curves, showing that a parameter space for these objects is obtained by the blow-up of $D_X$ described in Lemma \ref{desing}. Recall that, if $X$ is a blow-up of a curve, we denote by $\widetilde{X}=\ol{X-\cup E}$, for $E$ running over the set of exceptional components.

\begin{Def}
Let $C$ be a stable curve with two smooth components. An \emph{enriched spin curve of $C$ supported on $X$} is given by $(X, L_1, L_2)$, where $X$ is a blow-up of $C$ at a $\underline{\text{proper}}$ subset of nodes and $L_i\in \Pic X$, for $i=1,2$, with $L_i|_E\simeq  \mathcal O_E(1)$ for every exceptional component $E$ and  $$(L_i|_{\widetilde{X}})^{\otimes 2}\simeq \omega_{\widetilde{X}}\otimes T_{C_i}\,,\, (L_1|_{\widetilde{X}})\otimes (L_2|_{\widetilde{X}})\simeq \omega_{\widetilde X}$$ where $T_{C_i}$ is a twister of $\widetilde{X}$ induced by $C_i$ and a general smoothing of $\widetilde{X}$, the same for $i=1,2$.  An isomorphism between $(X, L_1, L_2)$ and $(X', L'_1, L'_2)$ is an isomorphism $\sigma\col X\ra X'$ commuting with the blow-up maps to $C$ and such that  
 $\sigma^*L'_i=L_i$ for $i=1,2$. Denote by $[X, L_1, L_2]$ the isomorphism class of an enriched spin curve, by $\ol{\SE_{C}}$ the set of the isomorphism classes of enriched spin curves of $C$, by $\SE_C$ the subset of the ones supported on $C$.
\end{Def}

For every set of indexes $I$, denote by $X_I$ the blow-up of $C$ at the nodes $\{p_i\}_{i\in I}$ of $C$. For a smooth curve $C$, denote by $J_2(C)$ the group of the two-torsion points of the Jacobian variety of $C$.

\begin{Prop}\label{tors}
Let $C$ be a curve with $\delta$ nodes and two smooth components of genus at least 1. Let $C^\nu$ be the normalization of $C$.
 Then, for every $I\subsetneq\{1,\dots,\delta\}$, the set of the isomorphism classes of enriched spin curves of $C$ supported on $X_I$ and $\SE_{\widetilde{X_I}}$ are isomorphic $J_2(C^\nu)\times (\ze/2\ze)^{\delta-|I|-1}\times(\co^*)^{\delta-|I|-1}$-torsors.
\end{Prop}

\begin{proof}
>From \cite[Lemma 2.1]{C}, a class $[X_I, L_1, L_2]$ is determined by the pair 
$(L_1|_{\widetilde{X_I}}, L_2|_{\widetilde{X_I}}$), hence the set of the isomorphism classes of enriched spin curves of $C$ supported on $X_I$ and $\SE_{\widetilde{X_I}}$ are in bijection. Thus, it suffices to show that $\SE_{\widetilde{X_I}}$ is a  $J_2(C^\nu)\times (\ze/2\ze)^{\delta-|I|-1}\times(\co^*)^{\delta-|I|-1}$-torsor. The set 
$\{(\widetilde{X_I}, \omega_{\widetilde{X_I}}\otimes T_{C_1}, \omega_{\widetilde{X_I}}\otimes T_{C_2})\}$ is in bijection with $\E_{\widetilde{X_I}}$, hence by Proposition \ref{Maino} it is a $(\co^*)^{\delta-|I|-1}$-torsor. By definition,  $L_1|_{\widetilde{X_I}}$ determines $L_2|_{\widetilde{X_I}}$. For $\omega_{\widetilde{X_I}}\otimes T_{C_1}$ fixed, the set of square roots of $\omega_{\widetilde{X_I}}\otimes T_{C_1}$ is a $J_2(C^\nu)\times  (\ze/2\ze)^{\delta-|I|-1}$-torsor, because $C^\nu$ is the normalization of $X_I$.
\end{proof}

>From Proposition \ref{tors}, we get a partition: 
$$\ol{\SE_C}=\cup_{I\subsetneq\{1,\dots,\delta\}}\SE_{\widetilde{X_I}}.$$

\begin{Rem}\label{smooth-bundle}
Let $f\col\mathcal X\ra B$ be a smoothing of a nodal curve $X$ and let $\mathcal N\in\Pic(\mathcal X).$  Let $L\in\Pic(X)$ and let $\iota_0$ be an isomorphism 
$\iota_0\col L^{\otimes 2}\ra\mathcal N\otimes\mathcal O_X.$ By \cite[Remark 3.0.6]{CCC}, up to shrinking $B$ to a complex neighbourhood of $0$, there exists $\mathcal L\in\Pic\mathcal X$ extending $L$ and an isomorphism $\iota\col\mathcal L^{\otimes 2}\ra\mathcal N$ extending $\iota_0.$ Moreover, if $(\mathcal L',i')$ is another extension of $(L,\iota_0),$ then there is an isomorphism $\chi\col\mathcal L\ra\mathcal L',$ restricting to the identity, with $\iota=\iota'\circ\chi^{\otimes 2}.$ 
\end{Rem}

Keep Notation \ref{DX} and the notation of Remark \ref{line-torsor}. 
Let $C$ be a stable curve with two smooth components and $\delta$ nodes with 
$\text{Aut}(C)=\{id\}$. 
Recall that $D_C=\text{Def}(C)\cap \co^\delta_{t_1,\dots,t_\delta}$ 
and $D_X=\phi^{-1}(D_C)$, where $\phi\col\Sgb\ra\Mgb$.
 Let $S_C^{sing}$ be the set of the spin curves of $C$ such that $D_X$ is singular.  
 Recall that $S_C^{sing}$ is described in Lemma \ref{desing}.   
 Now, $S_C^{sing}$ is a $J_2(C^\nu)$-torsor, where $C^\nu$ is the normalization of $C$, then $\cup_{\xi\in S^{sing}_C}(\mb{P}^{\delta-1}_\xi-\cup_{1\le i\le \delta} H_{\xi,i})$ is a $J_2(C^\nu)\times(\ze/2\ze)^{\delta-1}$ $\times(\co^*)^{\delta-1}$-torsor, where $\mb{P}^{\delta-1}_\xi$ and $H_{\xi,i}$ are as in Remark \ref{line-torsor}.

\begin{Thm}\label{Th1}
Let $C$ be a curve with $\delta$ nodes and two smooth components of genus at least 1. Assume that  
$\A(C)=\{id\}$. Let $C^\nu$ be the normalization of $C$.  Then
$\SE_C$ and $\cup_{\xi\in S^{sing}_C}(\mb{P}^{\delta-1}_\xi-\cup_{1\le i\le \delta} H_{\xi,i})$ are isomorphic $J_2(C^\nu)\times(\ze/2\ze)^{\delta-1}$ $\times(\co^*)^{\delta-1}$-torsors. 
\end{Thm}

\begin{proof}
Let $C_1,C_2$ be the components of $C$. Pick $(C, L_1, L_2)\in \SE_C$.  
By the definition of $\SE_C$, there exists a general smoothing $f\col \C\ra B$ of $C$ such that $L_i^{\otimes 2}\simeq\omega_C\otimes T_{C_i}$ for $i=1,2$, where $T_{C_1}$ (resp. $T_{C_2}$) is the twister induced by $C_1$ (resp. $C_2$) and $\C$.
Let $D_C$ be as in Notation \ref{DX}. 
Let $R\subset D_C$ be the line through the origin, away from the coordinate hyperplanes, such that, up to restrict $B$, the induced map $B\ra\D(C)$ has $R$ as image. By Proposition \ref{Maino}, the line $R$ does not depend on the chosen smoothing $\C\ra B$.

Set $\ol{S}_f(\omega_f):=\Sgb\times_{B}\Mgb$. 
Pick the $B$-curves $\ol{S}_f(\omega_f(C_i))$, for $i=1,2$, as in \cite[Theorem 2.4.1]{CCC}. Recall that the fiber of $\ol{S}_f(\omega_f(C_i))\ra B$ over $0\in B$ represents limit square roots of $\omega_f(C_i)|_C$, for $i=1,2$. Notice that $L_i$ is a limit square roots of $\omega_f(C_i)|_C$ for $i=1,2$.  
Let $\ell_i\in \ol{S}_f(\omega_f(C_i))$ be the point representing $L_i$. Since $\omega_f(C_i)$ and $\omega_f$ are isomorphic away from the special fiber, the curves $\ol{S}_f(\omega_f(C_i))$  and $\ol{S}_f(\omega_f)$ are isomorphic away from the fiber over $0\in B$. This implies that they have the same normalization $S^\nu_f$. 
Call: $$\psi\col S^\nu_f\ra\ol{S}_f(\omega_f)$$ the normalization. By \cite[4.1]{CCC},  $\ol{S}_f(\omega_f(C_i))$  is smooth at $\ell_i$ for $i=1,2$. Therefore $S^\nu_f$ and 
$\ol{S}_f(\omega_f(C_1))$ (resp. $S^\nu_f$ and 
$\ol{S}_f(\omega_f(C_2))$) are isomorphic locally at $\ell_1$ (resp. locally at $\ell_2$). In particular, we can regard $\ell_1$ and $\ell_2$ as points of $S^\nu_f$.

We are able to describe $\psi(\ell_1)$ and $\psi(\ell_2)$. Set $C_1\cap C_2=\{p_1,\dots, p_\delta\}$. By definition, $L_1\simeq L_2\otimes T_{C_1}$. Let $\xi=(X, G)\in S^{sing}_C$ be a spin curve of $C$, where $X$ is the blow-up of $C$ at the whole set of nodes and $G$ is given by the following data, for every exceptional component $E$ of $X$:
\begin{equation}\label{restriction}
G|_{E}\simeq\mathcal O_E(1) \, , \, G|_{C_i}\simeq (L_i)|_{C_i}\simeq (L_{3-i})|_{C_i}(-\sum_{1\le s\le\delta}p_s) \text{ for }i=1,2.
\end{equation}
Take the Cartesian diagram:
\[
\SelectTips{cm}{11}
\begin{xy} <16pt,0pt>:
\xymatrix{
\X\ar[r] & \C' \ar[r] \ar[d] &  \C \ar[d]^f\\
& B' \ar[r]^{g} & B
}
\end{xy}
\]
where $g$ is the degree 2 covering of $B$, totally ramified over $0$, and $\X$ is the blow-up at the nodes of $C$, so that $\X$ is a smoothing of $X$. Call $\pi\col\X\ra \C$ the composed map. Let $\mathcal{L}_1$ (resp. $\mathcal L_2$) be the 
line bundle of $\C$ such that $\mathcal L_1|_C\simeq L_1$ and $\mathcal L_1^{\otimes 2}\simeq \omega_f\otimes T_{C_1}$ (resp. $\mathcal L_2|_C\simeq L_2$ and $\mathcal L_2^{\otimes 2}\simeq \omega_f\otimes T_{C_2}$), as in Remark \ref{smooth-bundle}. Of course, $\mathcal L_1\simeq \mathcal L_2\otimes T_{C_1}$. Set $\mathcal G_i:=\pi^*\mathcal L_i\otimes \mathcal O_{\X}(C_{3-i})$, for $i=1,2$. By construction, for every exceptional component $E\subset X$ we have:
$$\mathcal G_i|_{\widetilde{X}}\simeq G|_{\widetilde{X}}\,,\,\mathcal G_i|_{E}\simeq G|_E\simeq\mathcal O_E(1).$$ This implies that $L_1$ (resp. $L_2$) is isomorphic to a line bundle $G_1$ (resp. $G_2$) in the isomorphism class of $\xi$. Therefore $L_1$ and $G_1$ (resp. $L_2$ and $G_2$) are limits of the same family of theta characteristics, hence $\ell_1,\ell_2\in \psi^{-1}(\xi)$. Since $\mathcal L_1\simeq\mathcal L_2\otimes T_{C_1}$, then also $L_1$ and $L_2$ are limits of the same family of theta characteristics, hence 
$\ell_1=\ell_2\in \psi^{-1}(\xi)\subset S^\nu_f$. 
Let $D_X^\nu\stackrel{\nu}{\ra}D_X\stackrel{\phi}{\ra}D_C$ be as in Remark \ref{line-torsor}. By construction, $\ol{S}_f(\omega_f)$ is given by $\phi^{-1}(R)$, locally at 
$\xi$. In particular, the strict transform $(\nu\circ\phi)^*(R)$ of $R$ is contained in $S^\nu_f$ and $\ell_1=\ell_2 \in\mb{P}^{\delta-1}_{\xi}-\cup_{1\le i\le \delta} H_{\xi,i}$. Define: 
$$\chi\col\SE_C\lra\cup_{\xi\in S^{sing}_C}(\mb{P}^{\delta-1}_\xi-\cup_{1\le i\le \delta} H_{\xi, i})$$ as $\chi(C, L_1, L_2):=\ell_1=\ell_2$.

We show that $\chi$ is surjective. Consider  
$\ell\in \mb{P}^{\delta-1}_\xi-\cup_{1\le i\le \delta} H_{\xi,i}$, where $\xi=(X, G)\in S^{sing}_C$. Let $R\subset D_C$ be the line corresponding to $\ell$. Then 
$\ell\in (\nu\circ\phi)^*(R)$, hence  
$\ell\in S^\nu_f$ and $\psi(l)=\xi.$ By \cite[Lemma 4.1.1]{CCC}, we have $|\psi^{-1}(\xi)|\le 2^{\delta-1}$. 
Being $G$ fixed, the data (\ref{restriction}) determine a set $\mathcal F_1$ (resp. $\mathcal F_2$)  of $2^{\delta-1}$ non-isomorphic line bundles represented by $2^{\delta-1}$ different smooth points of $\ol{S}_f(\omega_f(C_1))$ (resp. $\ol{S}_f(\omega_f(C_2))$). 
Thus $|\psi^{-1}(\xi)|=2^{\delta-1}$ and the subset of  $\ol{S}_f(\omega_f(C_1))$
(resp. $\ol{S}_f(\omega_f(C_2))$) 
representing $\mathcal F_2$ (resp. $\mathcal F_2$) is $\psi^{-1}(\xi)$. In particular, $\ell$ represents two line bundles  $L_1, L_2$ appearing in a enriched spin curve and $\chi(C, L_1, L_2)=\ell$.

We show that $\chi$ is injective. Assume that $\chi(C, L_1, L_2)=\chi(C, L'_1, L'_2).$ In particular, if $\ell_i$ and $\ell'_i$ are the points of $S^\nu_f$ representing $L_i$ and $L'_i$, for $i=1,2$, then $\ell_i=\ell'_i$, which implies that $L_i\simeq L'_i$, for $i=1,2$.
\end{proof}

\begin{Thm}\label{Th2}
Let $C$ be a curve with $\delta$ nodes and two smooth components of genus at least 1. Assume that  
$\A(C)=\{id\}$. Then for every $\emptyset\ne I\subsetneq\{1,\dots,\delta\}$ we have that $\cup_{\xi\in S^{sing}_C}(\cap_{i\in I}H_{\xi,i}-\cup_{i\notin I}H_{\xi,i})$ and $\SE_{\widetilde{X_I}}$ are isomorphic $J_2(C^\nu)\times(\ze/2\ze)^{\delta-|I|-1}\times(\co^*)^{\delta-|I|-1}$-torsors. 
\end{Thm}

\begin{proof}

\emph{First step}.
Without loss of generality, let $I=\{1,\dots, h\}$ and $X_I$ be the blow-up of $C$ at the first $h$ nodes. From now on, $\xi=(X, G)\in S^{sing}_C$ will be a fixed spin curve of $C$, where $X$ is the blow-up of $C$ at the whole set of nodes. Pick $\ell\in \cap_{i\in I}H_{\xi,i}-\cup_{i\notin I}H_{\xi,i}$. Let $D_C$ be as in Notation \ref{DX}. Now, $\ell$ corresponds to a line of $D_C$ with parametrization: 
$$(0, 0, \dots, 0, t_{h+1}, \alpha_{h+2}t_{h+1}, \alpha_{h+3}t_{h+1},\dots, \alpha_\delta t_{h+1}),$$ for some $\alpha_i\in \co^*$. 
Consider the curve $R\subset D_C$ with parametrization:
\begin{equation}\label{B-curve}
(t_{h+1}^2,\dots, t_{h+1}^2, t_{h+1},\alpha_{h+2}t_{h+1}, \alpha_{h+3}t_{h+1},\dots, \alpha_\delta t_{h+1}).
\end{equation} 
Let $f\col\C\ra B$ be a smoothing of $C$ such that, up to restrict $B$, the induced map $B\ra \D(C)$ has $R$ as image.
Notice that $\ell$ is contained in the strict transform $(\nu\circ\phi)^*(R)$ of $R$.  Locally at the first $h$ nodes of $C$, the surface $\C$ is given by $\{xy-t_{h+1}^2=0\}\subset\co^3_{x, y, t_{h+1}}$. Let $\X_I\ra \C$ be the resolution of this singularities. The special fiber of $h\col \X_I\ra B$ is $X_I$ and $\X_I$ is smooth. Pick the $B$-curve $\ol{S}_f(\omega_f)=\ol{S_g}\times_B \Mgb$ and its normalization: $$\psi\col S^\nu_f\ra\ol{S}_f(\omega_f).$$  Now, $\ol{S}_f(\omega_f)$ is given by $\phi^{-1}(R)$, locally at $\xi$, and as in Theorem \ref{Th1}, the strict transform $(\nu\circ\phi)^*(R)$ of $R$ is contained in $S^\nu_f$. 
In particular $\ell\in\psi^{-1}(\xi)$.

\smallskip

\emph{Second Step}. Consider the smoothing $h\col\X_I\ra \C$ of $X_I$. For $i=1,2$, pick the $B$-curves $\ol{S}_{h}(\omega_{h}(C_i))$, as in \cite[Theorem 2.4.1]{CCC}, which are isomorphic to $\ol{S}_f(\omega_f)$ away from the special fiber. The fiber of $\ol{S}_{h}(\omega_{h}(C_i))\ra B$ over $0\in B$ represents limit square roots of $(X_I, \omega_{h}(C_i)|_{X_{I}})$.  In the Second Step, we define points $\ell_i\in\ol{S}_{h}(\omega_{h}(C_i))$ such that $\ell_i\in \psi^{-1}(\xi)$, for $i=1,2$.

Pick the following limit square roots of $(X_I, \omega_{h}(C_i)|_{X_{I}})$.   Let $E_1,\dots, E_h$ be the exceptional components of $X_I$. For $i=1,2$, let $Y_i$ be the blow-up of $X_I$ at the nodes $C_{3-i}\cap E_1,\dots, C_{3-i}\cap E_h$ and call $F_{3-i,1},\dots, F_{3-i,h}$ the new exceptional components, as in Figure 1. Set $\{p_{h+1},\dots, p_\delta\}:=C_1\cap C_2$.
\[
\begin{xy} <16pt,0pt>:
(0,0)*{\scriptstyle}="a"; 
"a"+(0,0);"a"+(-0.5,-1)**\dir{-}; 
"a"+(0,0);"a"+(-0.5,1)**\dir{-}; 
"a"+(0,0);"a"+(2,0)**\crv{"a"+(1,2)}; 
"a"+(0,0);"a"+(2,0)**\crv{"a"+(1,-2)}; 
"a"+(2,0);"a"+(4,1)**\crv{"a"+(3,2)}; 
"a"+(2,0);"a"+(4,-1)**\crv{"a"+(3,-2)};
"a"+(2.5,1.5);"a"+(4.5,-0.5)**\dir{-}; 
"a"+(2.5,-1.5);"a"+(4.5,0.5)**\dir{-};
"a"+(-0.8,1.2)*{\scriptstyle{C_i}}; 
"a"+(-1.3,0)*{Y_i}; 
"a"+(-0.8,-1.2)*{\scriptstyle{C_{3-i}}};  
"a"+(5,0.6)*{\scriptstyle{F_{3-i,1}}}; 
"a"+(4.8,-0.8)*{\scriptstyle{E_1}};  
"a"+(1.7,-2.4)*{\text{Figure 1}};  
\end{xy}
\] 
 Let $(Y_i, L_i)$, for $i=1,2$, be a limit square root of  
$(X_I, \omega_{h}(C_i)|_{X_{I}})$ defined by the conditions  
$L_i|_{E_j}\simeq \mathcal O_{E_j}\,,\, L_i|_{F_{3-i,j}}\simeq\mathcal O_{F_{3-i,j}}(1)$, for $1\le j\le h$, and: 
\begin{equation}\label{restriction2}
L_i|_{C_i}\simeq G|_{C_i}\,,\,L_i|_{C_{3-i}}\simeq G|_{C_{3-i}}(\sum_{h< s\le \delta} p_s).
\end{equation} 
Let $\ell_i$ be the point of $\ol{S}_{h}(\omega_{h}(C_i))$ representing $(Y_i, L_i)$, $i=1,2$. Since $1 \le h<\delta$,  the graph 
$\Sigma_{Y_i}$ has one node and $h$ loops, $\ol{S}_{h}(\omega_{h}(C_i))\ra B$ is \'etale at $\ell_i$, $i=1,2$, by \cite[4.1]{CCC}. Thus $\ol{S}_{h}(\omega_{h}(C_i))$ and $S^\nu_f$ are isomorphic, locally at $\ell_i$ and we will show that $\ell_i\in \psi^{-1}(\xi)$, $i=1,2$. Take the Cartesian diagram: 
\[
\SelectTips{cm}{11}
\begin{xy} <16pt,0pt>:
\xymatrix{
 \Z \ar[d]^{\pi_3}\ar[dr]^{\pi_2} \ar[r]^{\pi_1} &  \Y_1 \ar[r] & \X'_I\ar[r]\ar[d]^{h'} & \X_I \ar[r] \ar[d]_h &  \C \ar[dl]_f\\
\X  &  \Y_2\ar[ur] & B' \ar[r]^{g} & B &  
}
\end{xy}
\]
where $g$ is the degree 2 covering of $B$, totally ramified over $0$, $\Y_i\ra \X'_I$ is the blow-up at the nodes  $C_{3-i}\cap E_1,\dots, C_{3-i}\cap E_h$ for $i=1,2$ and $\Z$ is blow-up at the remaining nodes of $X_I$. We will specify the map $\pi_3$ later.
Notice that $Y_i$ is the special fiber of $\Y_i\ra B'$ for $i=1,2$ and $\Z$ is smooth. Denote by $Z$ the special fiber of $\Z\ra B'$ and let  $F_{i1},\dots,F_{ih}, E_h\dots, E_\delta$ be the exceptional components of $\pi_i$, for $i=1,2$, as in Figure 2. 
\[
\begin{xy} <16pt,0pt>:
(0,0)*{\scriptstyle}="a"; 
"a"+(-3.5,-1.5);"a"+(3.5,-1.5)**\crv{"a"+(0,1)};
"a"+(-3.5,1.5);"a"+(3.5,1.5)**\crv{"a"+(0,-1)};
"a"+(2.5,1.5);"a"+(4,0.2)**\dir{-}; 
"a"+(0,1);"a"+(0,-1)**\dir{-}; 
"a"+(-2,1.3);"a"+(-2,-1.3)**\dir{-}; 
"a"+(2.5,-1.5);"a"+(4,-0.2)**\dir{-};
"a"+(3.6,-1);"a"+(3.6,1)**\dir{-};
"a"+(-3.8,1.2)*{\scriptstyle{C_i}}; 
"a"+(-5.3,0)*{Z}; 
"a"+(-4,-1.2)*{\scriptstyle{C_{3-i}}};  
"a"+(2.7,-2)*{\scriptstyle{F_{3-i,1}}}; 
"a"+(2.6,1.8)*{\scriptstyle{F_{i,1}}}; 
"a"+(0,-1,4)*{\scriptstyle{E_2}}; 
"a"+(-2,1.6)*{\scriptstyle{E_3}}; 
"a"+(3.2,0)*{\scriptstyle{E_1}};  
"a"+(0,-3)*{\text{Figure 2}};  
\end{xy}
\] 
Let $\rho_i\col Y_i\ra X_I$ be the blow-up map and  $\mathcal{L}_i\in\Pic(\Y_i)$ be such that:
\begin{equation}\label{smoothLi}
\mathcal {L}_i|_{Y_i}\simeq L_i\,,\,\mathcal {L}_i^{\otimes 2}\simeq \omega_{\mathcal \Y_i/B'}(-\sum_{1\le j\le h} F_{3-i,j})\otimes \rho_i^*(\mathcal O_{\X_I}(C_i)|_{X_I})
\end{equation} 
as in Remark \ref{smooth-bundle}, for $i=1,2$. The second condition of (\ref{smoothLi}) comes from the very definition of limit square root. Let $\pi_3\col\Z\ra \X$ be the contraction of $F_{ij}$ for $i=1,2$ and $j=1,\dots, h$. In particular, the special fiber of $\X$ is the blow-up $X$ of $C$ at the whole set of its nodes.  For $i=1,2$, define:  
\begin{equation}\label{G}
\mathcal G_i:=(\pi_3)_*(\pi_i^* \mathcal L_i\otimes \mathcal O_{\Z}(C_{3-i}+\sum_{1\le j\le h}F_{3-i,j})).
\end{equation}
By construction, $\pi_i^* \mathcal L_i\otimes \mathcal O_{\Z}(C_{3-i}+\sum_{1\le j\le h}F_{3-i,j})$ has degree $0$ on each $F_{ij}$, hence $\mathcal G_i$ restricts to a line bundle on $X$. Furthermore, $\mathcal G_i|_{\widetilde{X}}\simeq G|_{\widetilde{X}}\,,\,\mathcal G_i|_{E}\simeq\mathcal O_{E}(1)$ for every exceptional component $E\subset X$. As in Theorem \ref{Th1}, we have that $L_i$ and a line bundle in the equivalence class of $\xi=(X, G)$ are limits of the same family of theta characteristics, hence $\ell_i\in\psi^{-1}(\xi)$, for $i=1,2$. 

\smallskip

\emph{Third Step}. In this Step we define an isomorphism:
$$\chi\col\cup_{\xi\in S^{sing}_C}(\cap_{i\in I}H_{\xi,i}-\cup_{i\notin I}H_{\xi,i})\lra\SE_{\widetilde {X_I}}.$$   
Now, $|\psi^{-1}(\xi)|\le 2^{\delta-h-1}$, by \cite[Lemma 4.4.1]{CCC}, and (\ref{restriction2}) define $2^{\delta-h-1}$ different limit square roots $(Y_1, L_1)$ (resp. 
$(Y_2, L_2)$) of $(X_I,\omega_h(C_1))$ (resp. $(X_I,\omega_h(C_2))$).  These limit square roots are represented by points of $\psi^{-1}(\xi)$. Hence  $|\psi^{-1}(\xi)|=2^{\delta-h-1}$ and each $l\in\psi^{-1}(\xi)$ represents a limit square root $(Y_1, L_1)$  of $(X_I,\omega_h(C_1))$ and a limit square root $(Y_2, L_2)$ of $(X_I,\omega_h(C_2))$. 
Define $\chi(\ell)=[\widetilde{X_I}, L_1|_{\widetilde{X_I}}, L_2|_{\widetilde{X_I}}].$ First of all, we show that $\chi(\ell)\in \SE_{\widetilde{X_I}}$. Set $q_{3-i,j}:=C_{3-i}\cap F_{3-i,j}\in C_{3-i}$ for $i=1,2$ and $1\le j\le h$. The definition of limit square root implies: 
$$(L_i|_{\widetilde{X_I}})^{\otimes 2}\simeq \omega_h(C_i)|_{\widetilde{X_I}}(-\sum_{1\le j\le h} q_{3-i, j}) \simeq\omega_{\widetilde{X_I}}\otimes \mathcal O_{\X_I}(C_i)|_{\widetilde{X_I}}(\sum_{1\le j\le h} q_{ij}),$$

for $i=1,2$. 
Set $M_i=\mathcal O_{\X_I}(C_i)|_{\widetilde{X_I}}(\sum_{1\le j\le h} q_{ij})$, for $i=1,2$. We have: 
$$M_1\otimes M_2\simeq \mathcal O_{\X_I}(C_1+C_2)|_{\widetilde{X_I}}(\sum_{1\le j\le h} (q_{ij}+q_{3-i,j}))\simeq\mathcal O_{\widetilde{X_I}}$$
$$M_i\otimes \mathcal O_{C_j}\simeq
\begin{array}{ll}
\begin{cases}
\mathcal O_{C_j}(-\sum_{h<s\le \delta} p_s) & i=j
\\
\mathcal O_{C_j}(\sum_{h<s\le\delta} p_s) & i\ne j
\end{cases}
\end{array}$$ 

for $i=1,2$. 
By Proposition \ref{enr-char}, $M_i$ is a twister $T_{C_i}$ of $\widetilde{X_I}$ induced by $C_i$ and a general smoothing of $\widetilde{X_I}$, the same for $i=1,2$. To prove the second condition of an enriched spin curve, take the families $\Y_1$ and $\Y_2$, which are the same family away from the special fibers. Let $\theta_i\col\Y\ra \Y_i$ be the blow-up of $\Y_i$ at $C_i\cap E_1,\dots, C_i\cap E_h$, for $i=1,2$, and call $Y$ its special fiber. Since $(Y_1, L_1)$ and $(Y_2, L_2)$ are represented by the same point of $S^\nu_f$, up to change $L_2$ in the isomorphism class of $(Y_2, L_2)$, we have that $L_1$ and $L_2$ are limits of the same family of theta characteristics, given by the line bundles $\mathcal L_i$ of (\ref{smoothLi}). 
Set $\mathcal N:=(\theta_1^*\mathcal L_1)\otimes (\theta_2^*\mathcal L_2)$ and $\mathcal N^*:=\mathcal N^*|_{\Y-Y}$.
Thus, $\mathcal N^*\simeq\omega_{\Y-Y/B'-0}$, hence $\mathcal N\simeq \omega_{\Y/B'}\otimes\mathcal O_{\Y}(D)$, where $D$ is a Cartier divisor supported on components of $Y$. By (\ref{restriction2}), 
$\mathcal N|_{C_i}\simeq\omega_{\Y/B'}\otimes\mathcal O_{C_i}(-\sum_{1\le j\le h}q_{ij})$, thus: 
$$(L_1\otimes L_2)|_{\widetilde{X_I}}\simeq\mathcal N|_{\widetilde{X_I}}\simeq\omega_{\Y/B'}\otimes \mathcal O_{\widetilde{X_I}}(-\sum_{1\le j\le h}(q_{ij}+q_{3-i, j}))\simeq \omega_{\widetilde{X_I}}.$$
Then, $[\widetilde{X_I}, L_1|_{\widetilde{X_I}}, L_2|_{\widetilde{X_I}}]\in\SE_{\widetilde {X_I}}$.

We conclude by showing that $\chi$ is a bijection. 
The injectivity of $\chi$ is trivial. In fact, if we give $(Y_i, L_i)$ and $(Y_i, L'_i)$ such that $L_i|_{\widetilde{X_I}}\simeq L'_i|_{\widetilde{X_I}}$, for $i=1,2$, then  
$(Y_i, L_i)$ and $(Y_i, L'_i)$ define the same limit square root, for $i=1,2$.   To show that $\chi$ is surjective, we show that the image of $\chi$ has the right cardinality. Indeed, an element of the image is determined by choosing $\xi$ in the set $S^{sing}_C$, which is a $J_2(C^\nu)$-torsor, by choosing $R\subset D_C$ in the set of curves with parametrization as in (\ref{B-curve}), which is a $(\co^*)^{\delta-h-1}$-torsor and $l$ in the set $\psi^{-1}(\xi)$, which is a $(\ze/2\ze)^{\delta-h-1}$-torsor. 
\end{proof}

\begin{Exa}\label{Example}
Consider  a stable curve $C$ with two smooth components $C_1,C_2$ and three nodes. Set $C_1\cap C_2=\{p_1,p_2,p_3\}$. Assume that $\text{Aut}(C)=\{id\}$. 
 For every  spin curve $\xi$ of $C$, let $\mb{P}^2_\xi, H_{\xi, 1}, H_{\xi, 2}, H_{\xi, 3}$ be as in Remark \ref{line-torsor}. Let
 $X_i$ be the blow-up of $C$ at $p_i$ with exceptional component $E_i$, for $i=1,2,3$ and let $X_{ij}$ the blow-up at $\{p_i,p_j\}$, with exceptional components $E_i, E_j$ for every $\{i,j\}\subset\{1,2,3\}$. 
  Let  $S_C^{sing}$ be the set of spin curves of Theorem \ref{Th1} and Theorem \ref{Th2}.  The set $\ol{\SE_C}$ of enriched spin curves of $C$ is stratified as shown in Figure 3, where $\xi$ runs over the set $S_C^{sing}$. 
\[
\begin{xy} <16pt,0pt>:
(0,0)*{\scriptstyle}="a"; 
(0.7,-3.3)*{\scriptstyle}="b"; 
(0.7,-7)*{\scriptstyle}="c"; 
(0.7,0)*{\scriptstyle}="d"; 
(0.7,-3.3)*{\scriptstyle}="g"; 
(3.5,-10.5)*{\scriptstyle\bullet}="b";
(0,-3.5)*{\scriptstyle}="f";
"f"+(3,-3.3);"f"+(3.4,-3.3)**\dir{-};
"f"+(3.5,-3.3);"f"+(4.5,-3.3)**\dir{-};
"f"+(4.6,-3.3);"f"+(5,-3.3)**\dir{-};
"f"+(3,-3.3);"f"+(3.4,-3.3)**\dir{-};
"f"+(3.5,-3.3);"f"+(4.5,-3.3)**\dir{-};
"f"+(4.6,-3.3);"f"+(5,-3.3)**\dir{-};
"f"+(3,-3.3);"f"+(3.4,-3.3)**\dir{-};
"f"+(3.5,-3.3);"f"+(4.5,-3.3)**\dir{-};
"f"+(4.6,-3.3);"f"+(5,-3.3)**\dir{-};
"f"+(3.2,-3.7);"f"+(4.2,-2.2)**\dir{--};
"f"+(4.9,-3.7);"f"+(3.8,-2.2)**\dir{--};
"a"+(3,-3.3);"a"+(5,-3.3)**\dir{--};
"a"+(3.2,-3.7);"a"+(4.2,-2.2)**\dir{--};
"a"+(4.9,-3.7);"a"+(3.8,-2.2)**\dir{--};
"a"+(2.5,-2);"a"+(5.5,-2)**\dir{-};
"a"+(2.5,-4);"a"+(5.5,-4)**\dir{-};
"a"+(5.5,-2);"a"+(5.5,-4)**\dir{-};
"a"+(2.5,-4);"a"+(2.5,-2)**\dir{-};
"a"+(2.5,-2);"a"+(5.5,-2)**\dir{-};
"a"+(2.5,-4);"a"+(5.5,-4)**\dir{-};
"a"+(5.5,-2);"a"+(5.5,-4)**\dir{-};
"a"+(2.5,-4);"a"+(2.5,-2)**\dir{-};
"a"+(-10,-12);"a"+(7,-12)**\dir{-};
"a"+(-10,0);"a"+(7,0)**\dir{-};
"a"+(-10,0);"a"+(-10,-12)**\dir{-};
"a"+(7,0);"a"+(7,-12)**\dir{-};
"a"+(-10,-1.7);"a"+(7,-1.8)**\dir{-};
"d"+(-6,-3);"d"+(-6.5,-4)**\dir{-}; 
"d"+(-6,-3);"d"+(-6.5,-2)**\dir{-}; 
"d"+(-2,-3);"d"+(-1.5,-4)**\dir{-}; 
"d"+(-2,-3);"d"+(-1.5,-2)**\dir{-}; 
"d"+(-6,-3);"d"+(-4,-3)**\crv{"d"+(-5,-1)}; 
"d"+(-6,-3);"d"+(-4,-3)**\crv{"d"+(-5,-5)}; 
"d"+(-4,-3);"d"+(-2,-3)**\crv{"d"+(-3,-1)}; 
"d"+(-4,-3);"d"+(-2,-3)**\crv{"d"+(-3,-5)}; 
"g"+(-6,-3);"g"+(-6.5,-4)**\dir{-}; 
"g"+(-6,-3);"g"+(-6.5,-2)**\dir{-}; 
"g"+(-6,-3);"g"+(-4,-3)**\crv{"g"+(-5,-1)}; 
"g"+(-6,-3);"g"+(-4,-3)**\crv{"g"+(-5,-5)}; 
"g"+(-4,-3);"g"+(-2,-2.5)**\crv{"g"+(-3,-1)}; 
"g"+(-4,-3);"g"+(-2,-3.5)**\crv{"g"+(-3,-5)};
"c"+(-2.9,-0.8);"c"+(-2.9,2.2)**\dir{-}; 
"c"+(-2.9,-1.5);"c"+(-2.9,-4.5)**\dir{-}; 
"c"+(-5.1,-1.5);"c"+(-5.1,-4.5)**\dir{-}; 
"c"+(-6,-2.5);"c"+(-4,-3)**\crv{"c"+(-5,-1)}; 
"c"+(-6,-3.5);"c"+(-4,-3)**\crv{"c"+(-5,-5)}; 
"c"+(-4,-3);"c"+(-2,-2.5)**\crv{"c"+(-3,-1)}; 
"c"+(-4,-3);"c"+(-2,-3.5)**\crv{"c"+(-3,-5)};
"a"+(-8.5,-3)*{\SE_C}; 
"a"+(-0.5,-4.3)*{\scriptstyle{C_2}}; 
"a"+(-6,-4.3)*{\scriptstyle{C_1}}; 
"a"+(-1,-6)*{\scriptstyle{C_2}}; 
"a"+(-6,-7.5)*{\scriptstyle{C_1}}; 
"a"+(-1,-9.7)*{\scriptstyle{C_2}}; 
"a"+(-5.5,-9.7)*{\scriptstyle{C_1}}; 
"a"+(-4.1,-10)*{\scriptstyle{E_i}};  
"a"+(-1.9,-10)*{\scriptstyle{E_j}};  
"a"+(-1.9,-6.3)*{\scriptstyle{E_i}}; 
"a"+(-8.5,-6.3)*{\SE_{\widetilde{X_i}}}; 
"a"+(-8.5,-10)*{\SE_{\widetilde{X_{ij}}}}; 
"a"+(3.3,-8)*{\cup_\xi(\scriptstyle{H_{\xi, i}-(H_{\xi,j}\cup H_{\xi,k}))}}; 
"a"+(4,-11.4)*{\cup_\xi(\scriptstyle{H_{\xi, i}\cap H_{\xi, j}})}; 
"a"+(3.7,-4.5)*{\cup_\xi(\scriptstyle{\mb{P}^2_{\xi}-(H_{\xi, 1}\cup H_{\xi, 2}\cup H_{\xi, 3})})}; 
"a"+(-8.5,-0.5)*{\text{Strata}}; 
"a"+(-8.5,-1.3)*{\text{of } \ol{\SE_C}}; 
"a"+(-3,-0.5)*{\text{Support of}}; 
"a"+(-3,-1.3)*{\text{enriched spin curves}}; 
"a"+(4,-0.5)*{\text{Strata}}; 
"a"+(4,-1.3)*{\text{of }\cup_\xi\mb{P}^2_{\xi}}; 
"a"+(-1.5,-12.5)*{\text{Figure 3}}; 
\end{xy}
\] 
\end{Exa}

\bigskip
\bigskip
\bigskip
Marco Pacini
\\
Universidade Federal Fluminense (UFF)
\\
Rua M\'ario Santos Braga S/N
\\
Niter\'oi-Rio de Janeiro-Brazil 
\\
email: pacini@impa.br


\begin{thebibliography}{}
\bibitem[C]{C} M. Cornalba, \emph{Moduli of curves and theta-characteristics}. Lectures on Riemann surfaces (Trieste, 1987), 560-589 World Sci. Publishing Teaneck, NJ 1989.
\bibitem[CC]{CC} L. Caporaso, C. Casagrande, \emph{Combinatorial properties of stable spin curves}. Communications in Algebra 2003 {\bf 31} n. 8, 3653--3672.
\bibitem[CCC]{CCC} L. Caporaso, C. Casagrande, M. Cornalba, \emph{Moduli of roots of line bundles on curves}. Trans. of the Amer. Math. Soc. 2007 {\bf 359}, 3733--3768.
\bibitem[EH]{EH} D. Eisenbud, J. Harris, \emph{Limit linear series: Basic theory}. Invent. Math. {\bf 85} 337--371, 1986.
\bibitem[EM]{EM} E. Esteves, N. Medeiros, \emph{Limit canonical systems on curves with two components}. Invent. Math., {\bf 149}, 267--338, 2002. 
\bibitem[L]{L} K. Ludwig, \emph{On the geometry of the moduli space of spin curves}. ArXiv:0707.1831v1.  
\bibitem[M]{M} L. Main\`o, \emph{Moduli Space of Enriched Stable Curves} 
Ph.D. Thesis, Harvard University, 1998.
\bibitem[P]{P} M. Pacini, \emph{Spin curves over non-stable curves}, Communications in Algebra 2008 {\bf 36}, 1365--1393.
\end{thebibliography}
\end{document}